\documentclass[12pt]{article}
\usepackage{amssymb,amsmath,amsfonts,amsthm}
\newcounter{parag}

\newtheorem{lem}{Lemma}
\newtheorem*{theorem}{Theorem}

\newtheorem{prop}{Proposition}

\begin{document}

$\mbox{   }$ $\mbox{   }$$\mbox{   }$$\mbox{   }$$\mbox{   }$$\mbox{   }$$\mbox{   }$$\mbox{   }$$\mbox{   }$$\mbox{   }$$\mbox{   }$$\mbox{   }$$\mbox{   }$$\mbox{   }$$\mbox{   }$$\mbox{   }$$\mbox{   }$$\mbox{   }$$\mbox{   }$$\mbox{   }$$\mbox{   }$$\mbox{   }$$\mbox{   }$$\mbox{   }$$\mbox{   }$$\mbox{   }$$\mbox{   }$$\mbox{   }$$\mbox{   }$$\mbox{   }$$\mbox{   }$$\mbox{   }$$\mbox{   }$$\mbox{   }$$\mbox{   }$$\mbox{   }$$\mbox{   }$$\mbox{   }$$\mbox{   }$$\mbox{   }$$\mbox{   }$$\mbox{   }$$\mbox{   }$$\mbox{   }$$\mbox{   }$$\mbox{   }$$\mbox{   }$$\mbox{   }$$\mbox{   }$$\mbox{   }$$\mbox{   }$$\mbox{   }$$\mbox{   }$$\mbox{   }$$\mbox{   }$$\mbox{   }$$\mbox{   }$$\mbox{   }$$\mbox{   }$$\mbox{   }$$\mbox{   }$$\mbox{   }$$\mbox{   }$$\mbox{   }$$\mbox{   }$$\mbox{   }$$\mbox{   }$$\mbox{   }$$\mbox{   }$$\mbox{   }$$\mbox{   }$$\mbox{   }$$\mbox{   }$$\mbox{   }$ MSC 20D60, 20D08, 20D99

\medskip
\medskip
\medskip

\begin{center}
{\bf {\large The group $J_4 \times J_4$ is recognizable by spectrum}}

\medskip
\medskip
\medskip

\medskip
\medskip
\medskip

 I.B. Gorshkov, N.V. Maslova

\end{center}

\bigskip

{\bf Abstract} {\it The spectrum of a finite group is the set of its element orders. In this paper we prove that the direct product of
two copies of the finite simple sporadic group $J_4$ is uniquely determined by its spectrum in the class of all finite groups.
\smallskip

 \smallskip
}
\bigskip

\section{Introduction}

Let $G$ be a finite group. Denote by $\omega (G)$ the {\it spectrum} of~$G$, i.\,e. the set of all element orders of $G$.
Recall that $G$ is {\it recognizable by spectrum} (or simply {\it recognizable}) if every finite group $H$ with $\omega(H)=\omega(G)$ is isomorphic to $G$. A finite group $L$ is {\it isospectral} to $G$ if $\omega(L)=\omega(G)$.

Denote by $\pi(G)$ the set of all prime divisors of the
order of $G$.
If $g \in G$, then denote $\pi(g)=\pi(\langle g \rangle)$. Let $$\sigma(G) = max\{|\pi(g)| \mid g\in G\}.$$

In 1994 W.~Shi \cite{Shi} proved that if a finite group $G$ has a non-trivial solvable normal subgroup, then there are infinitely many finite groups whose are isospectral to $G$. Moreover, in 2012 V.~D.~Mazurov and W.~Shi \cite{MazurovShi} proved that there are infinitely many finite groups isospectral to a finite group $G$ if and only if there is a finite group $L$ such that $L$ is isospectral to $G$ and the solvable radical of $L$ is non-trivial. Thus, the socle of a recognizable finite group is a direct product of nonabelian simple groups. At the moment, for many finite nonabelian simple groups and their automorphism groups, it was proved that they are recognizable (see, for example, \cite{VasBig}). In 1997 V.~D.~Mazurov \cite[Theorem~2]{Mazurov3} proved that the direct product of two copies of the group $Sz(2^7)$ is recognizable by spectrum. In this paper we prove the following theorem.

\begin{theorem} The direct product of two copies of the finite simple sporadic group $J_4$ is recognizable by spectrum.

\end{theorem}

Note that if the direct product of $k$ copies of a finite group $G$ is recognizable by spectrum, then for each $i \le k$ the direct product of $i$ copies of $G$ is recognizable by spectrum. Thus, the following problems are of interest.

\medskip

\noindent{\bf Problem 1.} {\it Let $G$ be a finite group which is recognizable by spectrum. What is the largest number $k=k(G)$ such that the direct product of $k$ copies of the group $G$ is still recognizable by spectrum}?

\medskip

\noindent{\bf Problem 2.} {\it Is it true that for each integer $k \ge 1$ there exists a finite simple group $G=G(k)$ such that the direct product of $k$ copies of $G$ is recognizable by spectrum}?

\medskip

In proving Theorem, we use the following assertion which is interesting in its own right.

\begin{prop}\label{rzr2}
{Let $G$ be a finite solvable group such that $\sigma(G)=2$ and for any $p,q\in\pi(G)$ the following conditions hold{\rm:}}

{$(1)$ $p$ does not divide $q-1${\rm;}}

{$(2)$ $pq \in \omega(G)$.}

{Then $|\pi(G)| \le 3$.}

\end{prop}

\noindent{\bf Remark.} The evaluation of Proposition~\ref{rzr2} is the best possible. Indeed, let $V_1$ and $H_1$ be the additive and the multiplicative groups of the field of order $3^{16}$, respectively,  $V_2$ and $H_2$ be the additive and the multiplicative groups of the field of order $81$, respectively. Assume that $H_i$ acts on $V_j$ by the following rules. Take $x\in H_i$ and $y \in V_j$. If $i=j$, then $x(y)=xy$. If $i \not = j$, then $x(y)=y$. Consider the group $G=(V_1\times V_2) \leftthreetimes (L_1\times L_2)$, where $L_1$ is the subgroup of order $17$ of $H_1$ and
$L_2$ is the subgroup of order $5$ of $H_2$. Then $\pi(G)=\{3,5,17\}$, $\sigma(G)=2$, and for any $p,q\in\pi(G)$, $p$ does not divide $q-1$ and $pq \in \omega(G)$.

\section{ Preliminaries}

Our terminology and notation are mostly standard and could be found in \cite{Atlas,Atlas2,Gorenstein}.

In this paper by ``group'' we mean ``a finite group''  and by ``graph'' we mean ``an
undirected graph without loops and multiple edges''.

Let $\pi$ be a set of primes. Denote by $\pi'$ the set of the primes
not in $\pi$. Given a natural $n$, denote by $\pi(n)$ the set of its
prime divisors.
A natural number $n$ with $\pi(n) \subseteq \pi$ is called a $\pi$-number.

Let~$G$ be a group.
Note that $\pi(G)$ is exactly $\pi(|G|)$. The spectrum of $G$ defines the {\it Gruenberg--Kegel} graph (or the {\it prime graph}) $GK(G)$ of G; in this graph the vertex set is $\pi(G)$, and different vertices $p$ and $q$ are adjacent in $GK(G)$ if and only if $pq$ is an element order of $G$.

A subgroup $H$ of a group $G$ is called a {\it Hall subgroup} if the numbers $|H|$ and $|G:H|$ are coprime. A group $G$ with $\pi(G) \subseteq \pi$ is called a $\pi$-group. A subgroup $H$ of a group $G$ is called a {\it $\pi$-Hall subgroup} if $\pi(H) \subseteq \pi$ and $\pi(|G : H|) \subseteq \pi'$. Note that $H$ is a $\pi$-Hall subgroup of a group $G$ if and only if $H$ is a Hall $\pi$-subgroup of $G$. We say that a finite group $G$ {\it has the property $E_\pi$} if $G$ contains a Hall $\pi$-subgroup. We say that a finite group $G$ {\it has the property $C_\pi$} if $G$ has the property $E_\pi$ and any two Hall $\pi$-subgroups of $G$ are conjugate in $G$. We denote by $E_\pi$ ($C_\pi$, respectively) the class of all groups $G$ such that $G$ has the property $E_\pi$ ($C_\pi$, respectively).

Recall that $Soc(G)$ and $F(G)$ denote
the socle (the subgroup generated by all the minimal non-trivial normal subgroups of $G$) and the Fitting subgroup (the largest nilpotent normal subgroup) of $G$, respectively.

For a prime $p$ and a $p$-group $G$, $\Omega_1(G)$ denotes the subgroup of $G$ generated by the set of all its elements of order $p$.

Recall that a group $H$ is a {\it section} of a group $G$ if there exist subgroups $L$ and $K$ of $G$ such that $L$ is normal in $K$ and $K/L\cong H$.

\begin{lem}[{\rm See \cite[Lemma~1]{Mazurov2} and \cite[Lemma~1]{Mazurov1}}]\label{LemmaMazurov}  Let a Frobenius group $H=F\rtimes C$ with kernel $F$ and cyclic complement $C=\langle c \rangle$ of order $n$ acts on a vector space $V$ of non-zero characteristic $p$ coprime to $|F|$. Assume that $F \not \le C_H(V)$. Then the correspondent semidirect product $V \rtimes C$ contains an element of order $pn$ and $dim \,C_V (\langle c \rangle) > 0$.
\end{lem}

\begin{lem}[{\rm See \cite[Lemmas 3.3,~3.6]{VasBig}}]\label{fact}
Let $s$ and $p$ be distinct primes, a group $H$ be a semidirect
product of a normal $p$-subgroup $T$ and a cyclic subgroup $C=\langle g\rangle$ of order $s$,
and let $[T, g]\neq 1$. Suppose that $H$ acts faithfully on a vector space $V$ of
positive characteristic $t$ not equal to $p$.

If the minimal polynomial of $g$ on $V$ equals to $x^s-1$, then $C_V(g)$ is non-trivial.

If the minimal polynomial of $g$ on $V$ does not equal $x^s-1$, then

$($i$)$ $C_T(g)\neq 1$;

$($ii$)$ $T$ is nonabelian;

$($iii$)$ $p=2$ and $s = 2^{2^{\delta}}+1$ is a Fermat prime.
\end{lem}

\begin{lem}[{\rm See, for example, \cite{Busarkin_Gorchakov}}]\label{fr}
Let $G=F \rtimes H$ be a Frobenius group with kernel $F$ and complement $H$. Then the following statements hold.

$(1)$ The subgroup $F$ is the largest nilpotent normal subgroup of $G$, and $|H|$ divides $|F| - 1$.

$(2)$ Any subgroup of order $pq$ from $H$, where $p$ and $q$ are {\rm(}not necessarily distinct{\rm)} primes, is
cyclic. In particular, any Sylow subgroup of $H$ is a cyclic group or a {\rm(}generalized{\rm)} quaternion
group.

$(3)$ If the order of $H$ is even, then $H$ contains a unique involution.

$(4)$ If the subgroup $H$ is non-solvable, then it contains a normal subgroup $S \times Z$ of index $1$ or $2$, where
$S\cong SL_2(5)$ and $(|S|, |Z|)=1$.

\end{lem}




\begin{lem}[{\rm \cite[Lemma~1]{HallP}}]\label{hall}
Let $G$ be a finite group and $\pi$ be a set of primes. If $G\in E_{\pi}$, then $S \in E_{\pi}$ for every composition factor $S$ of $G$.
\end{lem}

\begin{lem}[{\rm See \cite{Gross1} and \cite{Gross2}}]\label{EpiCpi} Let $\pi$ be a set of primes such that $2 \not \in \pi$.  Then $E_\pi=C_\pi$.

\end{lem}


\begin{lem}\label{keller2}
Let $H$ be a finite solvable group such that $\sigma(H)=1$. Then $|\pi(H)| \le 2$.
\end{lem}

\begin{proof} Follows directly from \cite[Theorem~1]{Higman}.
\end{proof}

\begin{lem}\label{keller}
Let $H$ be a finite solvable group such that $\sigma(H)=2$. Then $|\pi(H)| \le 5$.
\end{lem}

\begin{proof} Follows directly from \cite[Theorem~1]{Zhang}.
\end{proof}

\begin{lem}[{\rm See \cite{Atlas}}]\label{spectra}
$(1)$  $|J_4|= 2^{21} \cdot 3^3 \cdot 5 \cdot 7 \cdot 11^3 \cdot 23 \cdot 29 \cdot 31 \cdot 37 \cdot 43 ${\rm;}

$(2)$ $\omega(J_4)$ consists from all the divisors of numbers from the set $\{16, 23, 24, 28, 29, 30, 31, 35, 37, 40, 42, 43, 44, 66\}${\rm;}

$(3)$ $\omega(J_4\times J_4)=\{x \mid x$  divides $lcm\,(a,b)$, where $a,b\in \omega(J_4)\}$.

\end{lem}

\begin{lem}\label{FactrosWithProp} Let $H$ be a finite simple group. Assume that the following conditions hold{\rm:}

$(i)$ $\pi(H) \subseteq \pi(J_4)${\rm;}

$(ii)$ $\omega(H) \cap \{9, 25\} = \varnothing${\rm;}

$(iii)$ $|\pi(H) \cap \{11, 23, 29, 31, 37, 43\}| \ge 2$.

Then one of the following statements holds{\rm:}

$(1)$ $H \cong L_2(23)$ and $\pi(H) \cap \{11, 23, 29, 31, 37, 43\} = \{11,23\}${\rm;}

$(2)$ $H \cong M_{23}$ and $\pi(H) \cap \{11, 23, 29, 31, 37, 43\} = \{11,23\}${\rm;}

$(3)$ $H \cong M_{24}$ and $\pi(H) \cap \{11, 23, 29, 31, 37, 43\} = \{11,23\}${\rm;}

$(4)$ $H \cong L_{2}(32)$ and $\pi(H) \cap \{11, 23, 29, 31, 37, 43\} = \{11,31\}${\rm;}

$(5)$ $H \cong U_{3}(11)$ and $\pi(H) \cap \{11, 23, 29, 31, 37, 43\} = \{11,37\}${\rm;}

$(6)$ $H \cong L_{2}(43)$ and $\pi(H) \cap \{11, 23, 29, 31, 37, 43\} = \{11,43\}${\rm;}

$(7)$ $H \cong J_4$ and $\{11, 23, 29, 31, 37, 43\} \subset \pi(H)$.
\end{lem}

\begin{proof} In view of \cite{Zav}, if $\pi(H) \subseteq \pi(J_4)$ and $|\pi(H) \cap \{11, 23, 29, 31, 37, 43\}| \ge 2$, then $H$ is one of the following groups{\rm:} $L_2(23)$, $M_{23}$, $M_{24}$, $Co_3$, $Co_2$, $L_2(32)$, $U_3(11)$, $L_2(43)$, $U_7(2)$, $L_2(43^2)$, $S_4(43)$, $J_4$.

If $H \in \{Co_3,Co_2\}$, then $9 \in \omega(H)$ in view of \cite{Atlas}.

If $H \cong U_7(2)$, then $9 \in \omega(H)$ in view of \cite[Corollary~3]{ButurlakinLin}.

If $H \cong L_2(43^2)$, then $25 \in \omega(H)$ in view of \cite[Corollary~3]{ButurlakinLin}.

If $H \cong S_4(43)$, then $25 \in \omega(H)$ in view of \cite[Corollary~2]{ButurlakinSympl}.
\end{proof}

\begin{lem}\label{FactrosWithProp2} Let $H$ be a finite simple group. Assume that the following conditions hold{\rm:}

$(i)$ $\pi(H) \subseteq \pi(J_4)${\rm;}

$(ii)$ $\omega(H) \cap \{9, 25\} = \varnothing${\rm;}

$(iii)$ $|\pi(H) \cap \{5, 23, 29, 37, 43\}| \ge 2$.

Then one of the following statements holds{\rm:}

$(1)$ $H \cong M_{23}$ and $\pi(H) \cap \{5, 23, 29, 37, 43\} = \{5,23\}${\rm;}

$(2)$ $H \cong M_{24}$ and $\pi(H) \cap \{5, 23, 29, 37, 43\} = \{5,23\}${\rm;}

$(3)$ $H \cong L_2 (29)$ and $\pi(H) \cap \{5, 23, 29, 37, 43\} = \{5,29\}${\rm;}

$(4)$ $H \cong U_{3}(11)$ and $\pi(H) \cap \{5, 23, 29, 37, 43\} = \{5,37\}${\rm;}

$(5)$ $H \cong J_4$ and $\{5, 23, 29, 37, 43\} \subset \pi(H)$.
\end{lem}

\begin{proof} In view of \cite{Zav}, if $\pi(H) \subseteq \pi(J_4)$ and $|\pi(H) \cap \{5, 23, 29, 37, 43\}| \ge 2$, then $H$ is one of the following groups{\rm:} $M_{23}$, $M_{24}$, $Co_3$, $Co_2$, $L_2(29)$, $U_3(11)$, $U_4(7)$, $U_7(2)$, $L_2(43^2)$, $S_4(43)$, $J_4$.

As in the proof of Lemma~\ref{FactrosWithProp} we exclude the following groups{\rm:} $Co_3$, $Co_2$, $U_7(2)$, $L_2(43^2)$, $S_4(43)$.
Moreover, $25 \in \omega(U_4(7))$ in view of \cite[Corollary~3]{ButurlakinLin}.

\end{proof}

\begin{lem}\label{SolvSubrg} Let $G$ be a group and $$1=G_n<G_{n-1}<\ldots<G_1<G_0=G$$ be a normal series in $G$. Let $\pi=\{p_1,\ldots p_m\}$ be a set of pairwise distinct primes such that $p_k \in \pi(G_{i_k}/G_{i_k+1})$ and $i_k \not = i_l$ if $k\not = l$. Then $G$ contains a solvable subgroup $H$ such that $\pi(H)=\pi$.
\end{lem}

\begin{proof} Without loss of generality we can assume that $m=n$ and $p_i \in G_{i-1}/G_i$.
Let $T$ be a Sylow $p_n$-subgroup of $G_{n-1}$. Using the Frattini argument we conclude that $G=N_G(T)G_{n-1}$. Now $N_G(T)/N_{G_{n-1}}(T)\cong G/G_{n-1}$ and in view of induction reasonings, $N_G(T)/T$ contains a solvable subgroup $H_1$ such that $\pi(H_1)=\{p_1, \ldots, p_{n-1}\}$. Thus, we conclude that $N_G(T)$ contains a solvable subgroup $H$ such that $\pi(H)=\pi$.

\end{proof}

\begin{lem}\label{SolvSubrgHall} Let $G$ be a group and $$1=G_n<G_{n-1}<\ldots<G_1<G_0=G$$ be a normal series in $G$. Let $\pi_1, \ldots \pi_n$ be sets of odd primes such that $\pi_k \subseteq \pi(G_{i_k}/G_{i_k+1})$  and $i_k \not = i_l$ if $k\not = l$. Assume that $G_{i_k}/G_{i_k+1} \in E_{\pi_k}$ for each $k$. Then $G$ contains a solvable subgroup $H$ such that $\pi(H)=\cup_{i=1}^n \pi_i$.
\end{lem}

\begin{proof} Without loss of generality we can assume that $\pi_i \subseteq G_{i-1}/G_i$.

Let $T$ be a Hall $\pi_n$-subgroup of $G_{n-1}$. In view of the Feit-Thompson theorem \cite{FeitThompson}, $T$ is solvable. In view of Lemma~\ref{EpiCpi}, we have $G_{n-1} \in C_{\pi_n}$. Thus, using the Frattini argument we conclude that $G=N_G(T)G_{n-1}$. Now $N_G(T)/N_{G_{n-1}}(T)\cong G/G_{n-1}$. In view of induction reasonings, the group $N_G(T)/N_{G_{n-1}}(T)$ contains a solvable subgroup $H_1$ such that $\pi(H_1)=\cup_{i=1}^{n-1} \pi_i$.

If $N_{G_{n-1}}(T)/T$ is solvable, then we consider the complete preimage $H_2$ of $H_1$ in $N_G(T)$. Note that $H_2$ is solvable. In view of the Hall theorem \cite[Theorem 6.4.1]{Gorenstein}, $H_2$ contains a Hall $(\cup_{i=1}^n \pi_i)$-subgroup~$H$. Note that in this case $\pi(H)=\cup_{i=1}^n \pi_i$.

Thus, we can assume that $N_{G_{n-1}}(T)/T$ is non-solvable. In view of the Feit-Thompson theorem \cite{FeitThompson}, $|N_{G_{n-1}}(T)/T|$ is even. Put $R=N_G(T)/T$ and $A=N_{G_{n-1}}(T)/T$.  Let $S$ a Sylow $2$-subgroup of $A$. Using the Frattini argument we conclude that $R=N_R(S)A$. Thus, $N_G(T)/N_{G_{n-1}}(T)=R/A \cong N_R(S)/N_A(S)$ and so, $N_R(S)/N_A(S)$ contains a solvable subgroup $H_2$ isomorphic to $H_1$. Note that in view of the Feit-Thompson theorem \cite{FeitThompson}, $N_A(S)$ is solvable. Let $H_3$ be the complete preimage of $H_2$ in $N_R(S)$. Note that $H_3$ is solvable. Thus, in view of the Hall theorem \cite[Theorem 6.4.1]{Gorenstein}, $H_3$ contains a Hall $(\cup_{i=1}^{n-1} \pi_i)$-subgroup $H_4$ and $\pi(H_4)= \cup_{i=1}^{n-1} \pi_i$. Let $H$ be the complete preimage of $H_4$ in $N_G(T)$. Note that $H$ is solvable and $\pi(H)=\cup_{i=1}^n \pi_i$.

\end{proof}

\begin{lem}[{\rm See \cite[Lemma~10]{GorshStar}}]\label{FactorCoclique} For a finite group $G$ take a coclique $\rho$ in $GK(G)$ with $|\rho| = 3$. Then the following claims hold{\rm:}

$(i)$ there exists a nonabelian composition factor $S$ of $G$ and a normal subgroup $K$ of $G$ such that $S \cong Inn(S) \unlhd \overline{G} = G/K \le Aut(S)$ and $|\pi(S) \cap \rho| \ge 2$.

$(ii)$ If $\rho'$ is a coclique in $GK(G)$ with $|\rho'| \ge 3$ and $|\pi(S) \cap \rho'| \ge 1$, then $|G|/|S|$ is divisible by at most
one element of $\rho'$. In particular, $|\pi(S) \cap \rho'| \ge |\rho'|-1$ and $S$ is a unique composition factor of $G$ with $|\pi(S) \cap \rho'|\ge 2$.

\end{lem}

\begin{lem}[{\rm See \cite[Lemma~10]{Zav1999}}]\label{ZavExt} Let $V$ be a normal elementary abelian subgroup of a group $G$. Put
$H=G/V$ and denote by $G_1= V \rtimes H$ the natural semidirect product. Then $\omega(G_1) \subseteq \omega(G)$.
\end{lem}

\begin{lem}\label{Extensions} If $G$ is an extension of an elementary abelian group $V$ with the group $H \cong J_4 \times J_4$, then $\omega(G) \setminus \omega(J_4\times J_4) \not = \varnothing$.
\end{lem}

\begin{proof} We can assume that $\pi(V) \subset \pi(J_4) = \{2,3,5,7,11,23,29,31,37,43\}$. In view of Lemma~\ref{ZavExt}, we can assume that $G=V \rtimes H$, where $H \le G$ and $H=H_1\times H_2$ for $H_1\cong H_2 \cong J_4$.

Assume that $V$ is a $2$-group. In view of \cite{Atlas2}, $H_2$ contains a subgroup isomorphic to $U_3(11)$. So, in view of \cite[Lemma~5]{Zav2006}, the subgroup $V \rtimes H_2$ contains an element of order $2\cdot 37$. Therefore there is an element $z\in H_2$ such that $|z|=37$ and $C_V(\langle z\rangle)$ is non-trivial. Let $V_1= C_V(\langle z\rangle)$. Note that $C_G(\langle z\rangle)$ contains a subgroup $V_1 \rtimes H_1$.  In view of \cite{Atlas2}, $H_1$ contains a Frobenius group $\langle x\rangle \rtimes \langle y \rangle$, where $|x|=29$ and $|y|=28$. In view of Lemma~\ref{LemmaMazurov}, the group $V_1 \rtimes H_1$ contains either an element of order $2 \cdot 29$ or an element of order $2 \cdot 28$. So, $G$ contains either an element of order $2\cdot 28 \cdot 37$ or an element of order $2\cdot 29 \cdot 37$. Thus, $\omega(G) \setminus \omega(J_4\times J_4) \not = \varnothing$.

Assume that $V$ is a $3$-group. In view of \cite{Atlas2}, $H_2$ contains a Frobenius group $\langle z\rangle \rtimes \langle t \rangle$, where $|z|=37$ and $|t|=3$. If $9 \in \omega (V\rtimes H_2)$, then  $\omega(G) \setminus \omega(J_4\times J_4) \not = \varnothing$. Suppose that $9 \not \in \omega (V\rtimes H_2)$. In view of Lemma~\ref{LemmaMazurov}, $V \le C_G(\langle z\rangle)$. So, $V \rtimes H_1 \le C_G(\langle z\rangle)$. In view of \cite{Atlas2}, $H_1$ contains a Frobenius group $\langle x\rangle \rtimes \langle y \rangle$, where $|x|=29$ and $|y|=28$. In view of Lemma~\ref{LemmaMazurov}, the group $V \rtimes H_1$ contains either an element of order $3 \cdot 29$ or an element of order $3 \cdot 28$. So, $G$ contains either an element of order $3 \cdot 28 \cdot 37$ or an element of order $3 \cdot 29 \cdot 37$. Thus, $\omega(G) \setminus \omega(J_4\times J_4) \not = \varnothing$.

Assume that $V$ is a $5$-group. In view of \cite{Atlas2}, $H_2$ contains a Frobenius group $\langle z\rangle \rtimes \langle t \rangle \le L_2(11) \le M_{22} \le H_2$, where $|z|=11$ and $|t|=5$. If $25 \in \omega (V\rtimes H_2)$, then  $\omega(G) \setminus \omega(J_4\times J_4) \not = \varnothing$. Suppose that $25 \not \in \omega (V\rtimes H_2)$. In view of Lemma~\ref{LemmaMazurov}, $V \le C_G(\langle z\rangle)$. So, $V \rtimes H_1 \le C_G(\langle z\rangle)$. In view of \cite{Atlas2}, $H_1$ contains a Frobenius group $\langle x\rangle \rtimes \langle y \rangle \le L_5(2) \le H_1$, where $|x|=31$ and $|y|=5$. In view of Lemma~\ref{LemmaMazurov}, the group $V \rtimes H_1$ contains either an element of order $25$ or an element of order $5 \cdot 31$. So, $G$ contains either an element of order $25$ or an element of order $5\cdot 11 \cdot 31$. Thus, $\omega(G) \setminus \omega(J_4\times J_4) \not = \varnothing$.

Assume that $V$ is a $7$-group. In view of \cite{Atlas2}, $H_2$ contains a Frobenius group $\langle z\rangle \rtimes \langle t \rangle$, where $|z|=43$ and $|t|=7$, and $H_1$ contains a Frobenius group $\langle x\rangle \rtimes \langle y \rangle$, where $|z|=29$ and $|t|=7$. Similar as above, we conclude that either $49 \in \omega(G)$ or $7\cdot 29 \cdot 43 \in \omega(G)$. Thus, $\omega(G) \setminus \omega(J_4\times J_4) \not = \varnothing$.

Assume that $V$ is an $11$-group. In view of \cite{Atlas2}, $H_2$ contains a Frobenius group $\langle z\rangle \rtimes \langle t \rangle$, where $|z|=23$ and $|t|=11$, and $H_1$ contains a Frobenius group $\langle x\rangle \rtimes \langle y \rangle$, where $|z|=43$ and $|t|=7$. Similar as above, we conclude that either $121 \in \omega(G)$ or either $7\cdot 11 \cdot 23 \in \omega(G)$ or $11 \cdot 23 \cdot 43 \in \omega(G)$. Thus, $\omega(G) \setminus \omega(J_4\times J_4) \not = \varnothing$.

Assume that $V$ is a $23$-group. In view of \cite{Atlas2}, $H_2$ contains a Frobenius group $\langle z\rangle \rtimes \langle t \rangle$, where $|z|=11$ and $|t|=5$, and $H_1$ contains a Frobenius group $\langle x\rangle \rtimes \langle y \rangle$, where $|z|=29$ and $|t|=28$. Similar as above, we conclude that $\omega(G)$ contains an element of one of the following orders{\rm:} $5\cdot 23 \cdot 28$, $5 \cdot 23 \cdot 29$, $11 \cdot 23 \cdot 28$, $11 \cdot 23 \cdot 29$. Thus, $\omega(G) \setminus \omega(J_4\times J_4) \not = \varnothing$.

Assume that $V$ is a $29$-group. In view of \cite{Atlas2}, $H_2$ contains a Frobenius group $\langle z\rangle \rtimes \langle t \rangle$, where $|z|=23$ and $|t|=11$, and $H_1$ contains a Frobenius group $\langle x\rangle \rtimes \langle y \rangle$, where $|z|=43$ and $|t|=7$. Similar as above, we conclude that $\omega(G)$ contains an element of one of the following orders{\rm:} $7\cdot 11 \cdot 29$, $7 \cdot 23 \cdot 29$, $11 \cdot 29\cdot 43$, $23 \cdot 29 \cdot 43$. Thus, $\omega(G) \setminus \omega(J_4\times J_4) \not = \varnothing$.

Assume that $V$ is a $p$-group, where $p \in \{31,37,43\}$. In view of \cite{Atlas2}, $H_2$ contains a Frobenius group $\langle z\rangle \rtimes \langle t \rangle$, where $|z|=23$ and $|t|=11$, and $H_1$ contains a Frobenius group $\langle x\rangle \rtimes \langle y \rangle$, where $|z|=29$ and $|t|=7$. Similar as above, we conclude that $\omega(G)$ contains an element of one of the following orders{\rm:} $7 \cdot 11 \cdot p$, $7 \cdot 23 \cdot p$, $11 \cdot 29\cdot p$, $23 \cdot 29 \cdot p$. Thus, $\omega(G) \setminus \omega(J_4\times J_4) \not = \varnothing$.

\end{proof}

\section{Proof of Proposition~\ref{rzr2}}

Let $G$ be a solvable group such that $\sigma(G)=2$ and for any $p,q\in\pi(G)$ the following conditions hold{\rm:}

{$(1)$ $p$ does not divide $q-1${\rm;}}

{$(2)$ $pq \in \omega(G)$.}

In view of the Hall theorem \cite[Theorem 6.4.1]{Gorenstein}, it's enough to prove that $|\pi(G)|\not=4$. Suppose to the
contradiction that $G$ is a group of the least order satisfying the conditions $(1)$ and $(2)$, and $|\pi(G)|=4$.

In view of condition $(1)$, any element of $\pi(G)$ is odd. Thus, $G$ does not contain a generalized quaternion group as its Sylow subgroup.
Moreover, for any Sylow $p$-subgroup $S$ of $G$, the subgroup $\Omega_1(S)$ is non-cyclic,  in particular, $S$ is non-cyclic. Otherwise there
is $g \in G$ such that $|g|=p$ and $|\pi(C_G(g))\setminus\{p\}|=3$. A contradiction to Lemma \ref{keller2}.

Let $H_1$ be a minimal normal subgroup of $G$. From solvability of $G$ it follows that $H_1$ is an elementary abelian $p_1$-group for some
$p_1 \in \pi(G)$. Let $T_1$ be a Hall $(\pi(G)\setminus \{p_1\})$-subgroup of $G$. Sylow subgroups of $T_1$ are non-cyclic. It follows from
Lemma \ref{fr} that  $p_1t_1 \in \omega (H_1T_1)$ for each $t_1 \in \pi(T_1)$. Thus, $G=H_1T_1$ in view of minimality of $G$.

Let $H_2$ be a minimal normal subgroup of $T_1$. Then $H_2$ is an elementary abelian $p_2$-group for some $p_2\in \pi(T_1)$. Let $T_2$ be a
Hall $(\pi(T_1)\setminus\{p_2\})$-subgroup of $T_1$. Since $p_2-1$ is not divisible by primes from $\pi(T_2)$, we see that $H_2$ is non-cyclic.
Otherwise, it follows from Lemma \ref{fr}  that any element of prime order from $T_2$ centralizes $H_2$ and $p_3p_4 \in \pi(T_2)$. So, there exists $g \in C_{T_1}(H_2)$ such that $|\pi(g)| \ge 3$.
Sylow subgroups of $T_1$ are non-cyclic, consequently, $p_2t_2 \in \omega (H_2T_2)$ for each $t_2 \in \pi(T_2)$. Since $H_2$ is non-cyclic, we have $p_1p_2 \in \omega(H_1H_2)$. Thus, $G=H_1H_2T_2$ in view of minimality of $G$.

Let $R=Soc(T_2)$. Suppose that $R$ is cyclic.

Suppose that there exists a non-trivial subgroup $L$ of $R$ such that $|L|=p_3 \in \pi(T_2)$
and $C_{H_1H_2}(L)$ is non-trivial. Then $L$ is a characteristic subgroup of $R$ and so, $L$ is normal in $T_2$.
Since $p_3-1$ is not divisible by $p_4$, we see that $\Omega_1(H_4)<C_G(L)$, where $H_4$ is a Sylow $p_4$-subgroup of $T_2$ with $p_4\neq p_3$. Note that $C_{H_1H_2}(L) \unlhd C_G(L)$. Since
the subgroup $\Omega_1(H_4)$ is non-cyclic, in view of Lemma~\ref{fr}, there exists an element $g \in C_G(L)$ such that $|\pi(g)| \ge 3$, a contradiction.  Thus, $L$ acts fixed-point free on $H_1H_2$. So, by the Thompson theorem \cite{Thompson} $H_1H_2$ is nilpotent.

Let $F=F(T_2)$. Suppose that $F$ is cyclic.

Note that $F\neq T_2$ since Sylow subgroups of $T_2$ are non-cyclic.
Let $F<K\unlhd T_2$ and $K/F$ is an elementary abelian $p_j$-group for some  $p_j\in \pi(T_2)$. Since $K\neq F$, a Sylow $p_j$-subgroup
of $K$ acts non-trivially on the Sylow $p_i$-subgroup of $F$, where $p_i\neq p_j$. Thus, an element whose order is a power of $p_j$ acts non-trivially on a cyclic $p_i$-subgroup, and $p_j$ does not divide $p_i-1$, a contradiction to Lemma~\ref{fr}.

So, $F$ is non-cyclic.
Consequently, there exists a non-cyclic Sylow $p_3$-subgroup $P_3$ of $F$, which is characteristic in $F$, and so, is normal in $T_2$. Consider the group $T_3=P_3H_4$, where $H_4$ is a Sylow $p_4$-subgroup of $T_2$. Suppose that $\Omega_1(H_4)<C_{T_3}(P_3)$. Since $p_2p_4 \in \omega(G)$, there is $g \in \Omega_1(H_4)$ such that a subgroup $C_{T_1}(g)\cap H_2$ is non-trivial. Note that $\Omega_1(H_4)$ is non-cyclic. Consequently, in view of Lemma~\ref{fr}, there exists an element $t \in C_{T_1}(g)$ such that  $|t|=p_2p_3p_4$. A contradiction. Thus, there is $h \in H_4$ acting non-trivially on $P_3$.

Suppose that there exists an element $h_3 \in P_3$ such that $h_3$ centralizes a subgroup $H_j$ for some $j \in \{1, 2\}$.
Since $H_1H_2=H_1\times H_2$,  $p_1p_4 \in \omega(G)$, and $p_2p_4 \in \omega(G)$, there exists $t \in H_j$ such that $|\pi(C_G(t))|=4$.
In view of Lemma~\ref{keller2}, there exists $t_1 \in C_G(t)$ such that $|\pi(t_1)|\ge 3$; a contradiction. Similar, if $h$ centralizes a subgroup $H_j$ for some $j \in \{1, 2\}$, then taking into account that $H_1H_2=H_1\times H_2$ and $P_3$ is non-cyclic, we conclude that there exists $t \in H_j$ such that $|\pi(C_G(t))|=4$ and receive a contradiction. Thus, the group $P_3\langle h\rangle$ acts faithfully both on $H_1$ and on $H_2$, and $[P_3,h]\not = 1$.

Since $p\neq 2$, with using Lemma \ref{fact} we obtain that for each $j \in \{1, 2\}$ the minimal polynomial of $h$ over $H_j$  is equal to $x^{p_4}-1$. Consequently, $C_{H_j}(h)\neq \{1\}$ for each $j \in \{1, 2\}$.  Thus, the intersections $C_G(h) \cap H_1$ and $C_G(h) \cap H_2$ are non-trivial and so, $p_1p_2p_4 \in \omega(G)$. A contradiction.

We have that $R$ is non-cyclic.

Let $H_3$ be a non-cyclic Sylow subgroup of $R$ and $\{p_3\}=\pi(H_3)$. Consider the subgroup $H=H_1H_2H_3H_4$ of $G$, where $H_1$, $H_2$, and $H_3$ are non-cyclic elementary abelian groups, and $H_4$ is a Sylow $p_4$-subgroup of $T_2$, where $p_4\not =p_3$.

Suppose that there exists $K_3 \leq H_3$ such that $|K_3| \geq p_3^2$ and $K_3<C_H(H_i)\leq C_G(H_i)$ for some $j \in \{1, 2\}$.
Since $H_3$ is elementary abelian, we have that $K_3$ is elementary abelian too. Moreover, there exists $g \in H_i$ such that $C_H(g) \cap H_j$ is
non-trivial, where $j \in \{1,2\} \setminus \{i\}$.  Consequently, $p_1p_2p_3 \in \omega(C_G(g)) \subseteq \omega(G)$; a contradiction.

Suppose that $\Omega_1(H_4)<C_G(H_3)$. Remind that $H_3$ is non-cyclic so, here exists $g \in H_3$ such that $C_G(g) \cap H_2$ is non-trivial. Thus, taking into account that $\Omega_1(H_4)$ is non-cyclic, we conclude that $p_4p_3p_2 \in \omega(C_G(g)) \subseteq
\omega(G)$. A contradiction. Thus, there exists $h \in \Omega_1(H_4)$ such that $h$ acts on $H_3$ non-trivially, i.\,e. $[H_3,h]\not=1$.

Suppose that there exists $l\in \Omega_1(H_4)$ such that $H_3<C_G(l)$. Then taking into account that $H_3$ is non-cyclic we conclude that $l$ acts fixed point free on $H_1H_2$. Consequently, the subgroup $H_1H_2$ is nilpotent by the Thompson theorem \cite{Thompson}. Since $H_3$ is abelian, we have $H_3=[h,H_3]\times C_{H_3}(h)$ in view of \cite[Theorem~5.2.3]{Gorenstein}. Moreover, $[h,H_3]\langle h
\rangle$ is a Frobenius group, and  $|[h,H_3]|\geq p_3^2$ in view of Lemma~\ref{fr} since $p_4$ does not divide $p_3-1$. In a similar way as before, we receive that the group $[h,H_3]\langle h\rangle$ acts non-trivially on both $H_1$ and on $H_2$. Therefore $C_G(h) \cap H_1$ is non-trivial and $C_G(h) \cap H_2$ is non-trivial in view of Lemma \ref{fact}. So, $p_1p_2p_4 \in \omega(G)$. A
contradiction. Thus, any element from $\Omega_1(H_4)$ acts non-trivially on $H_3$.

Since $\Omega_1(H_4)$ is non-cyclic, there exists $m\in \Omega_1(H_4)$ such that $C_G(m) \cap H_3$ is non-trivial. Consider subgroups $H_i([H_3,m]\langle
m\rangle)$, where $H_i \in \{H_1,H_2\}$. Note, $[m,H_3]\langle m \rangle$ is a Frobenius group and $|[m,H_3]|\geq p_3^2$, so, $[H_3,m]$ acts
non-trivially on $H_i$. In view of Lemma \ref{fact}, $C_G(m) \cap H_i$ is non-trivial. Thus, $|\pi(C_G(m))|=4$.  Consequently, in view of Lemma~\ref{keller2}, there exists an element $u \in C_G(m)$ such that $|\pi(u)| \ge 3$. A contradiction. \qed\medskip

\section{Proof of Theorem}
Let $G$ be a finite group such that $\omega(G)=\omega(J_4\times J_4)$. The spectrum of $G$ could be found in Lemma~\ref{spectra}.

Put $$\pi_1=\{5, 11, 23, 29, 31, 37, 43\},$$ $$\pi_2=\{7, 11, 23, 29, 31,
37, 43\},$$ and $$\pi=\pi_1\cup\{7\}=\pi_2\cup\{5\}=\pi_1\cup\pi_2.$$

\begin{lem}\label{PKTRJ4J4} Let $H$ be a section of $G$ such that $\pi(H) \subseteq \pi_i$ for some $i\in \{1,2\}$. Then $\sigma(H) \le 2$.

\end{lem}

\begin{proof} Follows directly from Lemma~\ref{spectra}.

\end{proof}

\begin{lem}\label{silcik}
Let $p \in \pi$ and $P \in Syl_p(G)$. Then $P$ is non-cyclic.
\end{lem}
\begin{proof}  Suppose to the contradiction that there exists $p \in \pi$ such that $P \in Syl_p(G)$ and $P$ is cyclic. Denote by $\theta$ a
set $\pi_i$ such that $p \in \pi_i$. Let $C=C_G(P)/P$.
 We have $\pi(C)=\pi(G)\setminus\{p\}$, and any two primes from $\theta\setminus\{p\}$ are no-adjacent in $GK(C)$ in view of Lemma~\ref{spectra}.

 Suppose that $C$ is solvable. In view of the Hall theorem \cite[Theorem 6.4.1]{Gorenstein}, there exists a $\theta$-Hall subgroup $C_1$ of $C$. Note that $\sigma(C_1)=1$ and
 $|\pi(C_1)|=6$, a contradiction to Lemma~\ref{keller2}.

Note that $|\pi(C)\cap \theta|=6$. In view of Lemma~\ref{FactorCoclique}, there exists a nonabelian composition factor $R$ of $C$ such that $5 \le |\pi(R)\cap\pi| \le 6$ and $\pi (R) \subseteq \pi(G)\setminus\{p\}=\{ 2, 3, 5, 7, 11, 23, 29, 31, 37, 43\}\setminus\{p\}$. In view of \cite{Zav}, there is no a finite nonabelian simple group $R$ satisfying these conditions. A contradiction.
  \end{proof}

 \medskip


In view of  Proposition~\ref{rzr2}, we have $G\not\in E_{\rho}$ for $\rho=\{29, 31, 37, 43\}$. From Lemma \ref{hall} and the Sylow theorems it follows that there exists a
composition factor $S$ of $G$ such that $|\{29, 31, 37, 43\}\cap\pi(S)|\geq2$ and $\pi(S)\subseteq\pi(J_4)$. Thus, $S \cong J_4$ in view of Lemma~\ref{FactrosWithProp}.

Let $G_1=G$ and $C_1=Soc(G)$. For $i\ge 2$ we put $G_i=G_{i-1}/C_{i-1}$ and $C_i=Soc(G_i)$.

Let $s$ be the minimal number such that $C_s$ contains a compositional factor of $G$ which is isomorphic to $S$.

\begin{lem}\label{onze}
We have $11\in\pi(|G|/|S|)$.
\end{lem}
\begin{proof}
Assume that $11\not\in\pi(|G|/|S|)$.
The group $S$ contains two conjugacy classes of elements of order $11$. Since $11\not\in\pi(|G|/|S|)$, $G$ does not contains more than two
conjugacy classes of elements of order $11$.

 If $G$ contains the only conjugacy class of elements of order $11$, then we receive a contradiction by the same way as in the proof of Lemma~\ref{silcik}. Thus, we can assume that in $G$ there are exactly two conjugacy classes of elements of order $11$.

Let $x,y\in G$ such that $|x|=|y|=11$ and $x\not\in y^G$. We have $\pi\subset
\pi(C_G(x))\cup\pi(C_G(y))$. Let $|\pi(C_G(x))\cap\pi_1|\geq|\pi(C_G(y))\cap\pi_1|$. Put $\theta=\pi(C_G(x))\cap\pi_1\setminus\{11\}$. We have $|\theta|\geq 3$ and the vertices from $\theta$ are pairwise non-adjacent in $GK(C_G(x))$. Therefore by Lemma~\ref{FactorCoclique} there exists a nonabelian composition factor $H$ of  $C_G(x)$ such that $H$ is simple and $|\pi(H)\cap\pi_1|\geq|\theta|-1$.

Put $r$ to be a number such that $H$ is a section of $C_r$.

 Suppose that $r \ge s$, and let $\tilde{x}$ be the image of $x$ in $G_s$. Let $\tilde{H}<G_s$ be a minimal preimage of $H$ in $G_s$  such that $\tilde{H} \le C_{G_s}(\tilde{x})$. In view of Lemma~\ref{spectra}, we have $C_S(\tilde{x})$ is a $\{2,3,11\}$-group and so, in view of \cite{Herzog}, $C_S(\tilde{x})$ is solvable. Thus, $\tilde{H} \not \le S$. Note that $S$ is characteristic in $C_s$ and therefore is normal in $G_s$. Now consider $R=S\tilde{H}$, which is a preimage of $H$ in $G_s$, and note that $R$ contains nonabelian composition factors isomorphic to $S$ and $H$. Consider the factor-group $R/C_R(S)$, which is isomorphic to a subgroup of $Aut(S)\cong J_4$, and note that the order of $C_R(S)$ is coprime to $11$. Thus, in view of the Jordan--Holder theorem, $R/C_R(S)\cong J_4$ and the group $C_R(S)$ contains a nonabelian composition factor which is isomorphic to $H$. Therefore $H$ is a composition factor of $C_{G}(y)$. We get that $|\theta|\geq 5$ and $|\pi(H)\cap\theta|\geq 4$.

Suppose that $r< s$. Let
$\bar{y}$ be the image of $y$ in $G_r$ and $C_r=T_1\times\ldots \times T_k$, where $T_i$ are simple groups. It is easy to see that there is $i$ such that $H$ is a section of $T_i$, and without loss of generality we can assume that $i=1$ and so, $\pi(H) \subseteq \pi(T_1)$.

Note that $\pi(T_1)\subseteq \pi(J_4)$, therefore in view of \cite{Zav}, we have $11\not \in\pi(Out(T_1))$. Moreover, $11 \not \in \pi(T_1)$ and so, $11\not \in \pi(Aut(T_1)$. Therefore $N_{\langle\bar{y}\rangle}(T_1)=C_{\langle\bar{y}\rangle}(T_1)$. If $\bar{y} \not \in C_{G_r}(T_1)$, then consider $$K=\langle T_1^w\mid w \in {\langle \bar{y}\rangle}\rangle.$$ It is easy to see that $C_K(\bar{y}) \cong T_1$. Thus, in any case $\pi(H)\subseteq \pi(C_{G_r}(\bar{y}))$ and so, $\pi(H)\subseteq \pi(C_{G}({y}))$.
It follows that $|\theta|\geq 5$ and $|\pi(H)\cap\theta|\geq 4$.

Now it is easy to see that $11\not\in\pi(|H|)$. In view of \cite{Zav}, there is no a finite
nonabelian simple group $H$ satisfying these conditions. A contradiction.
\end{proof}

Let us prove that there exists a composition factor $T\neq S$, such that $T \cong J_4$.

From Lemma \ref{silcik} it follows that a Sylow $p$-subgroup is not cyclic for any $p\in \pi_1\cup\pi_2$. Therefore
$(\pi \setminus\{11\})\subseteq \pi(|G|/|S|)$, and Lemma~\ref{onze} implies  that $11\in \pi(|G|/|S|)$. Thus, $\pi \subseteq \pi(|G|/|S|)$.

In view of Lemmas~\ref{SolvSubrg}, \ref{keller}, and~\ref{PKTRJ4J4}, there exists a composition factor $T_1$ of $G$ such that $T_1\neq S$ and at least two primes from the set $\{11, 23, 29, 31, 37, 43\}$ divide $|T_1|$. In view of Lemma~\ref{FactrosWithProp}, $T$ is isomorphic to one of the following groups{\rm:} $L_2(23)$, $M_{23}$, $M_{24}$, $L_{2}(32)$, $U_{3}(11)$, $L_{2}(43)$, $J_4$.

Assume that $T_1$ is isomorphic to one of the groups $L_2(23)$, $M_{23}$, $M_{24}$. In view of \cite{Atlas}, $T_1$ contains a subgroup isomorphic to $23:11$ which is a Hall $\{11,23\}$-subgroup of $T_1$, therefore the corresponding chief factor of $G$ containing $T_1$ belongs to $E_{\{11,23\}}$. Thus, in view of Lemmas~\ref{SolvSubrgHall}, \ref{keller}, and~\ref{PKTRJ4J4}, we conclude that there exists a composition factor $T\neq S$, such that at least two primes from the set $\{29, 31, 37, 43\}$ divide $|T|$. In view of Lemma~\ref{FactrosWithProp}, we conclude that $T \cong J_4$.

Assume that $T_1$ is isomorphic to the group $U_{3}(11)$. In view of \cite{Atlas}, $T_1 \in E_{\{5,11\}}$, therefore the corresponding chief factor of $G$ containing $T_1$ belongs to $E_{\{5,11\}}$. Thus, in view of Lemmas~\ref{SolvSubrgHall}, \ref{keller}, and~\ref{PKTRJ4J4}, we conclude that there exists a composition factor $T\neq S$, such that at least two primes from the set $\{23, 29, 31, 43\}$ divide $|T|$. In view of Lemma~\ref{FactrosWithProp}, we conclude that $T \cong J_4$.

Assume that $T_1$ is isomorphic to the group $L_{2}(43)$. In view of \cite{Atlas2}, $T_1 \in E_{\{7,43\}}$, therefore the corresponding chief factor of $G$ containing $T_1$ belongs to $E_{\{7,43\}}$. Thus, in view of Lemmas~\ref{SolvSubrgHall}, \ref{keller}, and~\ref{PKTRJ4J4}, we conclude that there exists a composition factor $T\neq S$, such that at least two primes from the set $\{23, 29, 31, 37\}$ divide $|T|$. In view of Lemma~\ref{FactrosWithProp}, we conclude that $T \cong J_4$.

Assume that $T_1$ is isomorphic to the group $L_{2}(32)$. Note that  $\pi(T_1) \cap \{5,7,23,29,31,37,43\} =\{31\}$ in view of \cite{Atlas}. Thus, in view of Lemmas~\ref{SolvSubrg}, \ref{keller}, and~\ref{PKTRJ4J4}, we conclude that there exists a composition factor $T_2\neq S$, such that at least two primes from the set $\{5, 23, 29, 37, 43\}$ divide $|T_2|$. In view of Lemma~\ref{FactrosWithProp2}, we conclude that $T_2$ is isomorphic to one of the following groups{\rm:} $M_{23}$, $M_{24}$, $L_2(29)$, $U_{3}(11)$, $J_4$. The cases when $T_2$ is isomorphic to $M_{23}$, $M_{24}$, or $U_{3}(11)$ were considered above. If $T_2 \cong L_2(29)$, then in view of \cite{Atlas}, $T_2 \in E_{\{7,29\}}$. Therefore the corresponding chief factor of $G$ containing $T_2$ belongs to $E_{\{7,29\}}$. Thus, in view of Lemmas~\ref{SolvSubrgHall}, \ref{keller}, and~\ref{PKTRJ4J4}, we conclude that there exists a composition factor $T\neq S$, such that at least two primes from the set $\{23, 31, 37, 43\}$ divide $|T|$. In view of Lemma~\ref{FactrosWithProp}, we conclude that $T \cong J_4$.

\bigskip

Assume that $T \le C_t$, where $t$ is the minimal number such that $C_t$ contains a compositional factor which is isomorphic to $J_4$ and distinct from $S$. Without loss of generality we can assume that $t \ge s$.

Assume that $t=s$. Note that $C_s$ contains no more than two distinct compositional factors of $G$ whose are isomorphic to $J_4$. Thus, $S\times T$ is a characteristic subgroup of $C_s$ and so,  $G$ has a chief factor isomorphic to $S \times T$.

Now assume that $t>s$. In this case $S$ is a characteristic subgroup of $C_s$ and so, $S$ is normal in $G_s$. We have $G_s=N_{G_s}(S)$, $C_{G_s}(S)$ is a normal subgroup in $G_s$, and $G_s/C_{G_s}(S)$ is isomorphic to a subgroup of $Aut(S) \cong J_4$. Thus, in view of the Jordan--Holder theorem, $T$ is a composition factor of $C_{G_s}(S)$. Moreover, $SC_{G_s}(S)=S \times C_{G_s}(S)$ is a normal subgroup of $G_s$.

Hence, in any case there exists a normal subgroup $H$ of $G$ such that $\overline{G}=G/H$ has a normal subgroup $\overline{A}$ isomorphic to $J_4\times J_4$.

\medskip

Let us prove that $\overline{G}\cong J_4\times J_4$.  It is easy to see that $C_{\overline{G}}(\overline{A})$ is trivial. Otherwise if a prime $p_1$ divides $|C_{\overline{G}}(\overline{A})|$, then the group $\overline{G}$ contains an element of order $p_1 p_2 p_3$, where $p_2$ and $p_3$ are primes from the set $\{29, 31, 37, 43\}$, and $|\{p_1, p_2, p_3\}|=3$. A contradiction.
Note that $Aut(A) \cong J_4 \wr C_2 \cong (J_4 \times J_4).2$. Thus, $\overline{G}$ is isomorphic to either $J_4\times J_4$ or $J_4 \wr C_2$. Now it is easy to see that $32 \in \omega(J_4 \wr C_2)$, and in view of Lemma~\ref{spectra}, we have $\overline{G} \cong J_4 \times J_4$.

\medskip

Assume that $H$ is non-trivial. If $H$ is solvable, then there exists a normal subgroup $H_1$ of $G$ such that $H_1 \le H$ and $H/H_1$ is elementary abelian. It is easy to see that $$\omega(J_4\times J_4)=\omega(G/H) \subseteq \omega(G/H_1) \subseteq \omega(G)=\omega(J_4\times J_4).$$ Thus, we obtain a contradiction to Lemma~\ref{Extensions}.

Assume that $H$ is non-solvable. In view of the Feit-Thompson theorem \cite{FeitThompson},  $|H|$ is even. Let $S$ be a Sylow $2$-subgroup of $H$. Using the Frattini argument we conclude that $G=N_G(S)H$ and so, $N_G(S)/N_H(S)\cong G/H \cong J_4\times J_4$. Note that $N_H(S)$ is a non-trivial solvable subgroup of $N_G(S)$. Moreover, $$\omega(J_4\times J_4) \subseteq \omega(N_G(S)) \subseteq \omega(G) = \omega(J_4\times J_4).$$ So, we receive a contradiction as above.
Thus, $G \cong J_4 \times J_4$.

\section{Acknowledgements}

The first author is supported by  Russian Foundation for Basic Research (project
18-31-20011).

\bigskip

Ilya~B. Gorshkov

Sobolev Institute of Mathematics SB RAS

Novosibirsk, Russia

E-mail address: ilygor8@gmail.com

\medskip

Natalia~V. Maslova

Krasovskii Institute of Mathematics and Mechanics UB RAS

Ural Federal University

Yekaterinburg, Russia

E-mail address: butterson@mail.ru

ORCID: 0000-0001-6574-5335


\bigskip
\begin{thebibliography}{1}

\bibitem{Busarkin_Gorchakov}
V. M. Busarkin and Yu. M. Gorchakov, {\it Finite Splittable Groups}, Moscow: Nauka, 1968 [in Russian].

\bibitem{ButurlakinSympl}
A.~A.~Buturlakin, {\it Spectra of finite symplectic and orthogonal groups}, Siberian Adv. Math., 21:3 (2011), 176--210.

\bibitem{ButurlakinLin}
A.~A.~Buturlakin, {\it Spectra of finite linear and unitary groups}, Algebra and Logic, 47:2 (2008), 91--99.


\bibitem{Atlas} J. H. Conway, R. T. Curtis, S. P. Norton, R. A. Parker, R. A. Wilson,
\emph{Atlas of finite groups}, Oxford: Clarendon Press, 1985.




\bibitem{FeitThompson} W.~Feit, J.~G.~Thompson. {\it Solvability of groups of odd order}, Pacific Journal of Mathematics. 13 (1963), 775--1029.



\bibitem{Gorenstein}
{Gorenstein~D.}, {\it Finite groups}, N. Y.: Harper and Row, 1968.

\bibitem{GorshStar}  I.~B.~Gorshkov, A.~M.~Staroletov, On groups having the prime graph as alternating and symmetric groups, Communications in Algebra, 2019, https://doi.org/10.1080/00927872.2019.1572167 .


\bibitem{Gross1}  F.~Gross, {\it On a conjecture of Philip Hall}, Proc. London Math. Soc., Ser. III, 52:3 (1986), 464--494.

\bibitem{Gross2} F.~Gross, {\it Conjugacy of odd order Hall subgroups}, Bull. London Math. Soc., 19: 4 (1987), 311--319.

\bibitem{HallP} Ph.~Hall, Theorems like Sylow’s,  Proc. London Math. Soc. (3), 6 (1956), 286--304.

\bibitem{Herzog}
M.~Herzog , {\it On finite simple groups of order divisible by three primes only}, J. Algebra, 10:3 (1968), 383--388.


\bibitem{Higman}  G. Higman, {\it Finite groups in which every element has prime power order},Journal of the London Mathematical Society, s1-32:3 (1957),335--342.


\bibitem{Mazurov2} {V.~D.~Mazurov}, {\it On the set of orders of elements of a finite group}, Algebra and Logic, 33 (1994), 49--55.

\bibitem{Mazurov1} {V.~D.~Mazurov}, {\it Characterizations of finite groups by sets of orders of their elements}, Algebra and Logic, 36:1 (1997), 23--32.

\bibitem{Mazurov3} {V.~D.~Mazurov}, {\it A characterizations of finite nonsimple groups by the set of orders of their elements}, Algebra and Logic, 36:3 (1997), 182--192.

\bibitem{MazurovShi} V.~D.~Mazurov, W.~J.~Shi, {\it A criterion of unrecognizability by spectrum for finite groups}, Algebra and Logic, 51:2 (2012), 160--162.

\bibitem{Shi} W. Shi, {\it The characterization of the sporadic simple groups by their element orders}, Algebra Colloq., 1:2 (1994), 159--166.



\bibitem{Thompson}  J.~Thompson, {\it Finite groups with fixed-point-free automorphisms of prime order},  Proceedings of the National Academy of Sciences of the United States of America, 45:4 (1959), 578--581.


\bibitem{VasBig} A. V. Vasil'ev, {\it On finite groups isospectral to simple classical groups}, J. Algebra, 423 (2015), 318--374.

\bibitem{Atlas2} R. A. Wilson, et. al., \emph{ATLAS of Finite Group Representations}, Queen Mary, University of London,  http://brauer.maths.qmul.ac.uk/Atlas/v3/ .


\bibitem{Zav} A. V. Zavarnitsine, {\it Finite simple groups with narrow prime spectrum}, Siberian Electronic Mathematical Reports, 6 (2009),
    1--12.

\bibitem{Zav1999} A.~V. Zavarnitsine, V.~D. Mazurov, {\it Element orders in coverings of symmetric and alternating groups}, Algebra and Logic, 38:3 (1999), 159--170.

\bibitem{Zav2006} A.~V.~Zavarnitsine, {\it Recognition of finite groups by the prime graph}, Algebra and Logic, 45:4 (2006), 220--231.


\bibitem{Zhang}  J. Zhang, {\it Arithmetical conditions on element orders and group structure}, Proceedings of the American Mathematical
    Society, 123:1 (1995), 39--44.


























\end{thebibliography}
\end{document}